\documentclass[a4paper,leqno,11pt]{amsart}

\usepackage[top=1.5in,bottom=1.3in,left=1.3in,right=1.3in,marginparwidth=1.5cm]{geometry}

\usepackage[all]{xy}

\usepackage{amsmath, amssymb, amsfonts, latexsym, mdwlist, amsthm}
\usepackage{mathrsfs}

\usepackage[colorlinks=true, citecolor=blue, urlcolor=blue, linkcolor=blue%, pagebackref
]{hyperref}

\usepackage{calrsfs} %Do not delete!
\usepackage{calligra}
\DeclareFontFamily{T1}{calligra}{}
\DeclareFontShape{T1}{calligra}{m}{n}{<-> s * [1.0] callig15}{}
\DeclareMathAlphabet{\mathcalligra}{T1}{calligra}{m}{n}
\DeclareFontFamily{OT1}{pzc}{}
\DeclareFontShape{OT1}{pzc}{m}{it}{<-> s * [1.30] pzcmi7t}{}
\DeclareMathAlphabet{\mathpzc}{OT1}{pzc}{m}{it}
\DeclareMathAlphabet{\pazocal}{OMS}{zplm}{m}{n}
\DeclareMathAlphabet\mathbfcal{OMS}{cmsy}{b}{n}

\DeclareMathAlphabet{\mathbbold}{U}{bbold}{m}{n}

\usepackage[english]{babel}
\usepackage[utf8x]{inputenc}
\usepackage{enumerate}

\usepackage{amsthm}
\newtheorem{theorem}{Theorem}[section]

\usepackage[colorinlistoftodos]{todonotes}

\newcommand\bGamma{{\boldsymbol{\Gamma}}}

\renewcommand\AA{{\mathbb A}} %Do not delete!
\newcommand\CC{{\mathbb C}}
\newcommand\DD{{\mathbb D}}
\newcommand\EE{{\mathbb E}}
\newcommand\GG{{\mathbb G}}
\newcommand\HH{{\mathbb H}}
\newcommand\PP{{\mathbb P}}

\newcommand\ZZ{{\mathbb Z}}

\newcommand\CCC{{\mathcal C}}

\newcommand\EEE{{\mathcal E}}
\newcommand\bFFF{\mathbfcal F}
\newcommand\FFF{{\mathcal F}}
\newcommand\GGG{{\mathcal G}}
\newcommand\HHH{{\mathcal H}}
\newcommand\III{{\pazocal I}}
\newcommand\JJJ{{\pazocal J}}
\newcommand\LLL{{\mathcal L}}

\newcommand\MMM{{\mathcal M}}

\newcommand\NNN{{\mathcal N}}
\newcommand\OOO{{\mathcal O}}
\newcommand{\PPP}{\mathcal{P}}
\newcommand{\PPPb}{\pazocal{P}}
\newcommand\SSS{{\mathcal S}}
\newcommand\TTT{{\mathcal T}}

\newcommand\YYY{{\mathcal Y}}
\newcommand\ZZZ{{\mathcal Z}}

\newcommand{\Ext}{\operatorname{Ext}\nolimits}
\newcommand{\END}{{\mathcal End}}
\newcommand{\EXT}{{\mathcal Ext}}
\newcommand{\hExt}{{\mathbbold{Ext}}}
\newcommand{\ev}{\operatorname{ev}\nolimits}
\newcommand{\Hilb}{\operatorname{Hilb}\nolimits}
\newcommand{\HOM}{{\mathcal Hom}}

\newcommand{\im}{\operatorname{im}\nolimits}

\newcommand\Pic{\operatorname{Pic}\nolimits}

\newcommand\Prym{\operatorname{Prym}\nolimits}

\newcommand\Res{\operatorname{Res}\nolimits}
\newcommand\Supp{\operatorname{Supp}\nolimits}
\newcommand{\Tr}{\operatorname{Tr}\nolimits}

\newcommand\bdot{{\raisebox{1pt}{\scalebox{0.4}{$\bullet$}}}}
\newcommand{\ud}{\hspace{0.3ex}\mathrm{d}\hspace{-0.1ex}}

\renewcommand\tilde[1]{\widetilde{#1}}

\newcommand\lra{{\longrightarrow}}

\newlength{\rrrr}

\newcommand\into{\hookrightarrow}

\newcommand{\isoto}{{\lra\hspace{-1.3 em}
\raisebox{ 0.6 ex}{$\textstyle\sim$}\hspace{0.8 em}}}

\newcommand{\mapstoo}[1]{\:
\xymatrix@1{\ar@{|->} [r] ^-{#1}&}\:}

\def\geq{\geqslant}

\usepackage{amsthm}
\newtheorem*{main-theo}{Main Theorem}

\theoremstyle{definition}
\newtheorem{defn}[theorem]{Definition}

\newtheorem{remark}[theorem]{Remark}

\begin{document}

\title[Symplectic moduli of sheaves on Poisson surfaces]{Symplectic moduli space of
1-dimensional sheaves on Poisson surfaces}

\author[I. Biswas]{Indranil Biswas}

\address{Department of Mathematics, Shiv Nadar University, NH91, Tehsil
Dadri, Greater Noida, Uttar Pradesh 201314, India}

\email{indranil.biswas@snu.edu.in, indranil29@gmail.com}

\author[D. Markushevich]{Dimitri Markushevich} 

\address{Univ. Lille, CNRS, UMR 8524 -- Laboratoire Paul Painlev\'e, F-59000 Lille, France}

\email{dimitri.markouchevitch@univ-lille.fr}



\begin{abstract}
We show that the Poisson structure on the smooth locus of a moduli space of 1-dimensional sheaves on a 
Poisson projective surface $X$ over $\CC$ is a reduction of a natural symplectic structure.
\end{abstract}

\maketitle

\selectlanguage{english}

\tableofcontents
\section*{Introduction}

Let $S$ be a complex projective surface with a movable anticanonical class, and let $\gamma$ be a
nonzero section of the anticanonical sheaf $\omega_S^{-1}=\OOO_S(-K_S)$.
Then $(S,\,\gamma)$ is a Poisson surface. It is known that any moduli space $\mathfrak M_S$ of simple
sheaves on $S$
also carries a Poisson structure. The Poisson structure on $\mathfrak M_S$,
which is denoted by $\PPP$, is induced by $\gamma$;
this $\PPP$ can be pointwise described as follows: We have the natural identification 
$T_{[\FFF]}\mathfrak M_S\,=\,\Ext^1(\FFF,\,\FFF)$ of the tangent space to $\mathfrak M_S$ at a point $[\FFF]$ 
representing a simple sheaf $\FFF$ on $S$. Therefore, by Serre duality the cotangent space $T^*_{[\FFF]}\mathfrak M_S$ 
is identified with $\Ext^1(\FFF,\,\FFF\otimes \omega_S)$. The Poisson structure on the cotangent space
at $[\FFF]$ is given by the skew-symmetric bilinear form 
\begin{equation}\label{poi}
\PPP\,:\,\Ext^1(\FFF,\,\FFF\otimes \omega_S)\times\Ext^1(\FFF,\,\FFF\otimes \omega_S)\,
\xrightarrow{\,\,\cup\,\,}\,\Ext^2(\FFF,\,\FFF\otimes \omega_S^2)
\end{equation}
$$
\xrightarrow{\,\,\otimes\gamma\,\,}\, \Ext^2(\FFF,\,\FFF\otimes \omega_S)
\,\xrightarrow{\,\,\Tr\,\,}\, H^2(S,\, \omega_S)\,=\,\CC.
$$
To describe the symplectic leaves of the Poisson structure $\PPP$ on
$\mathfrak M_S$, assume that the zero locus $D\,=\,{\rm Div}(\gamma)$ of $\gamma$ is
smooth; then it is a union of at most two elliptic curves. Let $\FFF$ be a simple sheaf on $S$. Assume that the restriction 
$\FFF\big\vert_D$ is torsion-free as an $\OOO_D$--module. Then the symplectic leaf $\SSS_{[\FFF]}$ of the
$\PPP$ passing through $[\FFF]$ is the locus of flat deformations $\FFF'$ of $\FFF$ which remain
simple and satisfies the condition that $\FFF'\big\vert_D\,\simeq\, \FFF\big\vert_D$.
This was proved for torsion-free sheaves in \cite{Bot-1} and extended to all simple sheaves in \cite{HM}, \cite{R}.

The symplectic foliation is one of the ways to associate symplectic varieties to a Poisson structure. The 
symplectic structures on the leaves of this foliation are restrictions of the Poisson structure of the 
ambient variety. The natural question is whether we can go the opposite way and represent a given 
Poisson structure as the reduction of a symplectic structure defined on a bigger variety. To be more
precise, we say that a Poisson variety $(Y,\,\PPP)$ is a {\em Poisson reduction} of a symplectic variety
$(M,\,\sigma)$ if there exists a Poisson submersion $(M,\,\sigma^{-1})\,\longrightarrow\, (Y,\,c\cdot\PPP)$
with totally isotropic fibers, where $\sigma^{-1}\,\in\, H^0(M,\, \bigwedge^2\TTT_M)$ denotes
the Poisson bivector, dual to the symplectic form $\sigma\,\in\, H^0(M,\,\Omega^2_M)$, and $c\,\in\, \CC$
is a constant. {Equivalently, one says that $(M,\,\sigma)$ is a {\em symplectic realization} of the
Poisson structure $(Y,\,\PPP)$.} Clearly, this is 
not always possible, for example, a nonzero and non-symplectic Poisson structure on a compact complex surface 
can't be obtained as a reduction of a symplectic structure.

Applied to our situation, the question can be 
formulated as follows: Can we add to the simple sheaves $\FFF$ parametrized by $\mathfrak M_S$ some additional 
structure, an enhancement or a decoration, in such a way that the moduli space of enhanced sheaves 
$\widetilde{\mathfrak M}_S$ becomes symplectic and the natural map $\widetilde{\mathfrak M}_S\, \longrightarrow\,
\mathfrak M_S$ that simply forgets the enhancement is a Poisson submersion with isotropic fibers?

In the present paper we give an affirmative answer to this question for certain moduli spaces of 1-dimensional
sheaves on the Poisson surface $S$. We consider the sheaves $\FFF$ supported on smooth curves
$C\,\subset \,S$, transversal to $D$, which are invertible when viewed as $\OOO_C$--modules. Such a sheaf can
be expressed in the form
$\FFF\,=\,j_{C*}\LLL$, where $\LLL\,\in\,\Pic^d(C)$ is a line bundle of degree $d$ on $C$ and
$j_C\,:\,C\,\into\, S$ denotes the inclusion map. When $C$ runs over a complete family of
smooth deformations of a given curve in $S$, the union of their Picard varieties of degree $d$ fills
an open subset of ${\mathfrak M}_S$; this open subset of ${\mathfrak M}_S$ is denoted by $\JJJ^d$.

We construct an enhancement of the sheaves $\FFF\,=\,j_{C*}\LLL$ such that the
corresponding enhanced moduli space is symplectic in three different ways. The construction, realized in Section \ref{prym}, goes
as follows: Let $\Delta\,=\,D\cap C$; take another copy $C'$ of $C$ and glue $C,\,C'$ together at the
points of $\Delta$ into a reducible normal crossing curve $\Gamma$ with two components. Then a line
bundle on $\Gamma$ can be thought of as the result of gluing together a pair
of line bundles $\LLL\, \longrightarrow\, C$ and $\LLL' \, \longrightarrow\, C'$. The gluing of $\LLL$
with $\LLL'$ is done by identifying the fibers of $\LLL,\, \LLL'$ at the $\delta$ points
of $\Delta\,=\,\{x_1,\, \cdots,\, x_\delta\}$.
We can scale each identification $\LLL\big\vert_{x_i}\,=\,\LLL'\big\vert_{x_i}$ by a nonzero scalar factor $\lambda_i
\,\in\,\CC^*$, thus the torus
$(\CC^*)^\delta$ acts on the family of gluings of $\LLL,\,\LLL'$; obviously, scaling all the identifications
simultaneously by the same factor 
$\lambda\,\in\,\CC^*$ produces an isomorphic line bundle on
$\Gamma$, so the action of $(\CC^*)^\delta$ factors through an action of $(\CC^*)^\delta/\CC^*\,\simeq\,
(\CC^*)^{\delta-1}$. The moduli space parametrizing the gluings of line bundles $\LLL$ on $C$ with the
corresponding inverse $\LLL^{-1}$ on $C'$ form the Prym variety $\Prym(\Gamma)$. We have the natural
short exact sequence
$$
1\, \longrightarrow\, (\CC^*)^\delta/\CC^*\, \longrightarrow\, 
\Prym(\Gamma) \, \longrightarrow\, \Pic(C)\, \longrightarrow\, 0,
$$
in which the surjection $\Prym(\Gamma) \, \longrightarrow\, \Pic(C)$ is the restriction of any line bundle on
$\Gamma$ to $C$. If we perform this construction for all the curves $C_u$ from a complete family
of deformations of $C$ inside $S$ which remain smooth and transversal to $D$, we obtain the family
$\Gamma_u$ of normal crossing curves and the associated relative Prym variety $\PPPb$, which is the union of
$\Prym(\Gamma_u)$ over all $u$. This $\PPPb$ is a symplectic enhanced moduli space, where
the enhancement attached to a line bundle
$\LLL$ on $C_u$ is a gluing datum on $\Delta$, or an isomorphism of restrictions
$\LLL\big\vert_{\Delta}\,\,\isoto\,\, \LLL^{-1}\big\vert_{\Delta}$. The set of enhancements of a given $\LLL$ is a torsor
under the torus $(\CC^*)^\delta/\CC^*\,\simeq\, (\CC^*)^{\delta-1}$ (see Theorem \ref{pr-th} and Remark \ref{any-d}).

In Section 2 we consider the family of 1-dimensional sheaves obtained as an open part of the relative Picard 
scheme of line bundles $\LLL$ of given degree on the curves $Z$ moving in the complete family of 
deformations of $C+D$. The said family of sheaves fills an open part of the moduli space $\mathfrak M_S$ 
carrying a Poisson structure, and we single out a symplectic leaf in it by framing $\LLL$ to a fixed line 
$\theta$ on $D$ by an isomorphism $\LLL\big\vert_D\,\simeq\, \theta$. According to this definition, if $\LLL$ is framed, 
then its support $Z$ is automatically reducible and has $D$ as an irreducible component. The thus obtained 
symplectic leaves are symplectic realizations of the Poisson structure on $\JJJ^d$, and the enhancement of 
$(C,\,L)\,\in\, \JJJ^d$ is realized by fixing isomorphisms $L\big\vert_{x_i}\,\simeq\, \theta\big\vert_{x_i}$
at all the points $x_i$ of $C\cap D$.

In both Sections 1 and 2 we assume the divisor $D$ of the Poisson bivector $\gamma$ to be smooth, and
in Section 2 the divisor $D$ is also assumed to be 
irreducible. In Section 3, we provide an example of a symplectic realization of $\JJJ^d$ in the 
case when $D$ is reducible and non-reduced, a sum of a component of multiplicity two and a finite number of 
smooth rational components. The surface $S$ in this example is the projectivization $\PP(\EE)$ of the rank
two vector bundle $\EE\,=\,\OOO_X\oplus K_X\otimes\OOO_X(\DD)$ on a smooth projective curve $X$, where $\DD
\,=\,x_1+\ldots+x_l$ 
is a reduced divisor on $X$. Considering the natural projection $p\,:\,S\,\longrightarrow\, X$, we have 
$D\,=\,2D_\infty+p^{-1}(x_1)+\ldots +p^{-1}(x_l)$, where $D_\infty$ is the cross-section of $p$ corresponding to 
the direct summand $\OOO_X$ of $\EE$. We make $C$ run over the open part of the linear system $|K_X^{\otimes 
r}\otimes\OOO_X(r\DD)|$ for some $r>0$, assuming $C$ is transversal to $D$. The Hitchin correspondence 
allows us to interpret the relative Picard $\JJJ^d$ of this family of curves $C$ as a moduli space of Higgs 
bundles on $X$, whose spectral curves are exactly the curves $C$, and the enhancement of the sheaves from 
$\JJJ^d$ is achieved by fixing trivializations of each line bundle $L\in\Pic^d(C)$ 
at the $rl$ points of $D\cap C$. This enhanced moduli space is symplectic.

\section{Symplectic structure on a relative Prymian}
\label{prym}

Let $(S,\,D)$ be as before a Poisson surface with an anticanonical
divisor. We will work with pure 1-dimensional sheaves on the doubling $Y$ of $(S,\,D)$
constructed as follows: Take another copy $(S',\,D')$ of $(S,\,D)$, and let $\rho\, :\, S\,\isoto\, S'$
be the isomorphism given by the identity map. Consider the simple normal crossing
surface $Y$ obtained by identifying $D$ with $D'$ using $\rho\big\vert_D$, so we have $S\cap S'\,=\,D$.
Denote by $\iota$ the involution on $Y$ that permutes the components: $\iota\big\vert_S\,=\,\rho$,
$\iota\big\vert_{S'}
\,=\,\rho^{-1}$. This surface $Y$ is Gorenstein and has trivial canonical class, meaning we have
$\omega_Y\,\simeq\,\OOO_Y$.

Let us fix a component $H$ of the Hilbert scheme of curves in $Y$ 
containing as an open subset the locus of curves of the form $C\cup C'$ such that:
\begin{enumerate}
\item $\rho^*[C']\,=\,[C]$ is fixed, where $[-]$ denotes the class of a curve in the integer cohomology
(or in the N\'eron--Severi group) of the surface, and

\item $C,\, C'$ are smooth, connected, they intersect $D$ transversely and $$C\cap D\,\,=
\,\,C'\cap D\,\,\neq\,\,\varnothing .$$
\end{enumerate}

Denote by $H_0$ this open subset of $H$ and consider the universal family of curves $\bGamma\,
\longrightarrow\, H_0$ over $H_0$. Its fiber over any $u\,\in\, H_0$ is the reducible nodal curve
$\Gamma_u\,=\,C_u\cup C'_u$, where $C_u\,\subset\, S$ and $C'_u\,\subset \,S'$, which has two smooth
components. For any integer $d$, let
\begin{equation}\label{en}
\JJJ^d\,\,=\,\,\mathbf{Pic}^{d,d}(\bGamma/ H_0)
\end{equation}
be the relative
Picard scheme of bidegree $(d,\,d)$ parametrizing the invertible sheaves on $\Gamma_u$ whose
restrictions to both the components are line bundles of degree $d$. When $d\,=\,0$, we obtain the family
$\JJJ\,=\,\JJJ^0$ whose fiber over any $u\,\in\, H_0$ is the connected component $J^0(\Gamma_u)$ of zero
of the generalized Jacobian $J(\Gamma_u)$. This $J^0(\Gamma_u)$ is a connected commutative algebraic group 
that fits in the short exact sequence
\begin{equation}\label{genJ}
1\,\longrightarrow\, (\CC^*)^{\delta-1}\,\longrightarrow\, J^0(\Gamma_u)
\,\longrightarrow\, J(C_u)\times J(C'_u)\,\longrightarrow\, 0,
\end{equation}
where $\delta\,=\,(C_u\cdot D)_S\,=\,\#(C_u\cap C'_u)$. Denoting by $g$ the
common genus of $C_u$ and $C'_u$, we
have $g\,=\,\frac{(C_u)_S^2-(C_u\cdot D)_S}{2}+1$ and $$\dim J^0(\Gamma_u)\,
\,=\,\,2g+\delta-1\,\,=\,\,(C_u)^2_S+1.$$
 
Mukai in \cite{Muk-1} showed that the moduli space $\mathfrak M_S$ of simple sheaves on a smooth projective
surface $S$ with trivial canonical class is smooth and carries an everywhere nondegenerate two-form
whose restriction to the tangent space $T_{[\FFF]}{\mathfrak M}_S\,=\,\Ext^1_S(\FFF,\,\FFF)$ is the special
case of the above pairing $\PPP$ in \eqref{poi} when $\omega_S$ is trivial. Since $\PPP$ is Poisson, this
2-form is closed and thus it
defines a symplectic structure on $\mathfrak M_S$. Its closedness was also proved, in different ways,
in \cite{Muk-2}, \cite{Ran}, \cite{Hu}. The first two of these papers prove the result only for
locally free sheaves. In fact, Ran works in a more general context: he proves that the trace pairing
\begin{equation}\label{tp}
\Ext^1(\FFF,\,\FFF)\times\Ext^1(\FFF,\,\FFF)\,\xrightarrow{\,\,\cup\,\,}\,\Ext^2(\FFF,\,\FFF)
\,\xrightarrow{\,\,\Tr\,\,}\, H^2(X,\, \OOO_X)\,\longrightarrow\,\CC
\end{equation}
composed with a linear map to $\CC$ (independent of $\FFF$)
defines a closed $H^2(X,\, \OOO_X)$-valued 2-form on the smooth locus of the moduli space $\mathfrak M_X$ of
locally free sheaves for any complex space $X$. The proof of \cite{Hu} handles the case of arbitrary
simple sheaves $\FFF$, but the base variety $X$ is assumed to be smooth. We need this kind of results
for the moduli space of 1-dimensional (and hence non-locally free) sheaves over the above singular surface $Y$.
It should be clarified that this case is not covered in the cited works.

\begin{theorem}
In the above setting, assume that $D$ is semi-ample, and $D\,\neq\, \varnothing$. Let $\mathfrak M_Y$ denote the moduli
space of simple sheaves on $Y$, and let $\mathfrak M^d_Y\,\subset\, \mathfrak M_Y$ be
the locus of invertible sheaves of bidegree $(d,\,d)$ on the reducible curves 
$\Gamma_u\,=\,C_u\cup C'_u$ with $u\,\in\, H_0$. Then the following statements hold:
\begin{enumerate}
\item $H_0$ is smooth of pure dimension $(C_u)^2_S+1$, and
$\JJJ^d$ is smooth over $H_0$ of relative dimension
$(C_u)^2_S+1\,=\,\dim H_0$ such that its fiber
over any $u\,\in\, H_0$ is isomorphic to $J^0(\Gamma_u)$.

\item The subset $\mathfrak M^d_Y$ of $\mathfrak M_Y$ is open
and it is contained in the smooth locus of $\mathfrak M_Y$. The natural map
$\JJJ^d\,\longrightarrow\,\mathfrak M_Y$ (see \eqref{en}) is an open embedding whose image
is $\mathfrak M^d_Y$.

\item $\mathfrak M^d_Y$ carries a symplectic structure $\sigma$ whose restriction to the tangent space at $[\FFF]\,\in\,
\mathfrak M^d_Y$ is given by the pairing \eqref{tp} (with $X\,=\,Y$), composed with an isomorphism
$H^2(Y,\, \OOO_Y)\,\isoto\,\CC$.

\item The support map $\mathfrak M^d\,\longrightarrow\, H_0$, defined by $[\FFF]\,\longmapsto\, \Supp\FFF$, is
a Lagrangian fibration for the symplectic structure $\sigma$ in statement (3).
\end{enumerate}
\end{theorem}

\begin{proof}
\textbf{(1).}\, The Picard functor is smooth when the characteristic
of the base field is zero, so we conclude that $\JJJ^d$ is smooth over $H_0$. To
see that $H_0$
is smooth over $\CC$, consider the short exact sequence of normal bundles
$$
0\,\longrightarrow\, \NNN_{C_u/S}(-\Delta_u)\,\longrightarrow\, \NNN_{\Gamma_u/Y}
\,\longrightarrow\, \NNN_{C'_u/S'}\,\longrightarrow\, 0,
$$
where $\Delta_u\,=\,C_u\cap C'_u$. Since we have $\NNN_{C_u/S}(-\Delta_u)\,\simeq\,\omega_{C_u}$ and
$\NNN_{\Gamma_u/Y}\,\simeq\, \omega_{\Gamma_u}$, it follows that
$$
h^0(\NNN_{\Gamma_u/Y})\,=\,h^0(\omega_{C_u})\, +\, h^0(\NNN_{C'_u/S'})\,=\,(C_u)^2_S+1.
$$
The subspace $H^0(\NNN_{C_u/S}(-\Delta_u))$ of $H^0(\NNN_{C_u/S})$ represents the infinitesimal deformations
of $C_u$ lying in the subscheme $\Hilb_S(-\Delta_u)$ of $\Hilb_S$ which parametrizes the curves $C$ on $S$
for which $\Delta_u\,\subset\, C$. By our assumption, $|K_S|\,=\,|-D|_S\,=\,\varnothing$, so by Severi's
semiregularity, all the infinitesimal deformations of $C_u$ in $S$ are in fact unobstructed. Hence all the elements
of $H^0(\NNN_{C_u/S}(-\Delta_u))$ lift to actual deformations of $C_u$ inside $\Hilb_S(-\Delta_u)$. By the
same reasoning, all the elements of $H^0(\NNN_{C'_u/S'})$ lift to deformations of $C'_u$ in $S'$,
so 
$$\dim H_0\,=\,h^0(\NNN_{C_u/S}(-\Delta_u))+h^0(\NNN_{C'_u/S'})\,=\,h^0(\NNN_{\Gamma_u/Y}),
$$ 
and all the infinitesimal deformations of $\Gamma_u$ in $Y$ are unobstructed. This implies that $H_0$
is a smooth equidimensional open subset of the Hilbert scheme of $Y$ of dimension $(C_u)^2_S+1$. Thus
$\JJJ^d$ is smooth over $\CC$, and $\dim\JJJ^d\,=\,2\dim H_0$.

\textbf{(2).}\, The existence of a classifying morphism $\JJJ^d\,\longrightarrow\,
\mathfrak M_Y$ follows from the fact that the Picard scheme represents, while $\mathfrak M_Y$ corepresents, the
respective moduli functors. Its image coincides with ${\mathfrak M}^d_Y$ by definition, and it is evidently a
bijection between $\JJJ^d$ and ${\mathfrak M}_Y^d$.
To see that the classifying morphism $\JJJ^d\,\longrightarrow\, \mathfrak M_Y$
is open and is an isomorphism onto the image, it suffices to verify that $\mathfrak M_Y$ is
smooth along $\mathfrak M_Y^d$, and this is equivalent to the statement that
\begin{equation}\label{ltg0}
\dim\Ext^1_Y(\FFF,\,\FFF)\,\,=\,\,2\dim H_0
\end{equation}
for all $[\FFF]\,\in\, \mathfrak M_Y^d$.

{}From the local-to-global spectral sequence we deduce the short exact sequence
\begin{equation}\label{ltg}
0\,\longrightarrow\, H^1(\END_{\OOO_Y}(\FFF))\,\longrightarrow\, \Ext^1_Y(\FFF,\,\FFF)
\,\longrightarrow\, H^0(\EXT^1_{\OOO_Y}(\FFF,\,\FFF))\,\longrightarrow\, 0.
\end{equation}
Here $\FFF\,=\,j_{u*}\LLL$ for a line bundle $\LLL$ on $\Gamma_u$, so $\END_{\OOO_Y}(\FFF)\,\simeq\,
j_{u*}\OOO_{\Gamma_u}$ and $$\EXT^1_{\OOO_Y}(\FFF,\,\FFF))\,\,\simeq\,\, j_{u*}\NNN_{\Gamma_u/Y}
\,\,\simeq \,\,j_{u*}\omega_{\Gamma_u},$$ hence we have
$$\dim \Ext^1_Y(\FFF,\,\FFF)\,=\,h^1(\OOO_{\Gamma_u})+h^0(\omega_{\Gamma_u})\,=\,2\dim J(\Gamma_u)
\,=\,2\dim H_0.$$ This proves \eqref{ltg0}.

\textbf{(3).}\, As in \cite{Muk-1}, the non-degeneracy of the pairing as in
\eqref{tp} follows immediately from its identification with 
the Serre duality pairing on $Y$. The closedness of the 2-form can be proven by extending the approach of 
\cite[Chapter 10]{Hu} to the class of simple normal crossing varieties, to which $Y$ belongs, but it would 
be fastidious to spell out all the details of such a generalization. We can circumvent this complication
by a use of the following deformation trick.

We will construct by hand a 1-parameter smoothing of $Y$. To this end, 
consider the 3-fold scroll $\PP\,=\,\PP(\EEE)$ obtained by the projectivization of the rank two vector bundle 
$\EEE\,=\,\OOO_S\oplus\omega_S$. Let $$\pi\,\,:\,\,\PP\,\, \longrightarrow\,\, S$$
be the natural projection. Let $\OOO(1)\,=\,\OOO_{\PP/S}(1)$ 
be the tautological sheaf, and let $y_0\,\in\, H^0(\PP,\,\OOO(1))$ and $y_1\,\in\, H^0(\PP,
\,\OOO(1)\otimes\pi^*\omega_S^{-1})$ be
the sections associated to the embeddings of the two factors of the direct sum into $\EEE$; let $F\,\in\, 
H^0(S,\,\omega_S^{-1})$ be a nonzero section vanishing on $D$. Then $Y$ can be realized as the zero locus 
${\rm Div}_{\PP}(s)$ of a section $s$ of a line bundle on $\PP$:
$$
Y\,=\,{\rm Div}_{\PP}(s),\ \ s\,\in\, H^0(\PP,\,\OOO(2)\otimes\omega_S^{-2}),\ \ s\,=\,y_1^2-y_0^2\pi^*F^2.
$$
Choosing a generic section $G\,\in\, H^0(S,\,\omega_S^{-2})$, we obtain the following 1-parameter family
$\mathcal Y\,=\,\{Y_t\}_{t\in \AA^1}$ smoothing out the singularities of $Y\,=\,Y_0$:
$$
Y_t\,=\,{\rm Div}_{\PP}(s_t),\ \ s_t\,\in\, H^0(\PP,\,\OOO(2)\otimes\omega_S^{-2}),\ \ s_t\,
=\,y_1^2-y_0^2\pi^*(F^2+tG).
$$
Let $B$ be an open neighborhood of $0$ in $\AA^1$ such that $Y_t$ is smooth for all $$t\,\in\, B^*
\,:=\,B\setminus\{0\}.$$
Then $Y_t$ is a smooth projective surface with trivial canonical bundle, and hence
it is either an abelian surface or it is a K3 surface for each $t\,\in\, B^*$. Let $p\,:\,\YYY_B\, \longrightarrow\, B$
be the restriction of $\YYY$ over $B$. Choose some ample class $h$ on $\PP$; it defines a relative
polarization on $\YYY_B/B$.
Consider the relative moduli space of stable sheaves $M^s_{\YYY_B/B}(P)$, as constructed in
\cite[Theorem 4.3.7]{Hu}, where $P$ is the Hilbert polynomial of the sheaves from $\mathfrak M_Y^d$ with respect to
the ample class $h$. Let $\FFF\,=\,j_{u*}\LLL$ be a sheaf from $\mathfrak M_Y^d$. Then the two biggest proper
subsheaves of $j_{u*}\LLL\,=\, \FFF$ are $j_{u*}(\LLL\big\vert_{C_u})(-\Delta_u)$ and $j_{u*}(\LLL\big\vert_{C'_u})(-\Delta_u)$,
and none of them is destabilizing. So we conclude that $\FFF$ is actually stable.

Thus $\mathfrak M_Y^d$ is an open subset in the fiber of $M^s_{\YYY_B/B}(P)$ over $0\,\in\, B$.
It is the relative Jacobian of the complete family of reducible curves $\Gamma_u$ on $Y_0$, and it
deforms to the neighboring fibers as the stable part of the relative Jacobian of a complete family
of curves on $Y_t$; we denote by $\mathfrak M_{Y_t}^d$ this
stable part of the relative Jacobian of the complete family of curves
on $Y_t$. As $t$ varies, these moduli spaces fill the open subset
$\mathfrak M_{\YYY_B/B}^d\,\subset\, M^s_{\YYY_B/B}(P)$.
The fact that the dimension of the fibers coincides with the dimension of $\Ext^1_{Y_t}(\FFF,\,\FFF)$
actually implies the smoothness of the structure morphism $\mathfrak M_{\YYY_B/B}^d
\, \longrightarrow\, B$. By \cite{BPS}, the constancy of $\dim \Ext^1_{Y_t}(\FFF,\,\FFF)$, and the flatness
of $\mathfrak M_{\YYY_B/B}^d/ B$, together imply that
the spaces $\Ext^1_{Y_t}(\FFF,\,\FFF)$ are the fibers of the relative Ext-sheaf
$\EXT^1_{\mathfrak M_{\YYY/B}^d/ B}(\bFFF,\,\bFFF)$, which can be identified with the relative tangent bundle
$\TTT_{\mathfrak M_{\YYY/B}^d/ B}$. Here $\bFFF$ is a universal sheaf, which exists locally in the analytic
or \'etale topology of $B$, but the relative Ext-sheaf does not depend on the choice of $\bFFF$, and it is
globally defined. Thus the 2-forms $\sigma_{t}$ on $\mathfrak M_{Y_t}^d$ glue together into a relative
holomorphic 2-form $\boldsymbol\sigma$
on $\mathfrak M_{\YYY/B}^d/ B$, defined by the relative version of the pairing \eqref{tp}.
This 2-form is closed
when restricted to $\mathfrak M_{Y_t}^d$ for any $t\,\in\, B^*$ by \cite[Chapter 10]{Hu}, and
consequently $\sigma\, :=\,\boldsymbol\sigma\big\vert_{\mathfrak M^d_{Y_0}}$ is also closed by continuity.
 
\textbf{(4).}\, This in fact follows from \eqref{ltg}. The subspace $$H^1(\END_{\OOO_Y}(\FFF))
\,\,\simeq\,\, H^1(\OOO_Y)$$ of $T_{[\FFF]}\mathfrak M^d_Y\,=\,
\Ext^1_Y(\FFF,\,\FFF)$ is nothing else but the tangent space to the fiber $J^d(\Gamma_u)$ of
$\mathfrak M^d_Y$ over $u$. Also, the Yoneda pairing restricts to it, up to the sign, as the composition
of maps $\EXT^0(\FFF,\,\FFF)
\times\EXT^0(\FFF,\,\FFF)\, \longrightarrow\,\EXT^0(\FFF,\,\FFF)$ combined with the cup product
$H^1(\EXT^0(\FFF,\,\FFF))\times H^1(\EXT^0(\FFF,\,\FFF))\, \longrightarrow\,
H^2(\EXT^0(\FFF,\,\FFF))$; but this combined map is 
zero because the support of $\EXT^0(\FFF,\,\FFF)$ is 1-dimensional, so $H^2(\EXT^0(\FFF,\,\FFF))
\,=\,0$. Thus the entire fiber $J^d(\Gamma_u)$ is totally isotropic
in $\mathfrak M^d_Y$ for the symplectic structure $\sigma$. Now the entire fiber $J^d(\Gamma_u)$ is
Lagrangian because its dimension coincides with the half of the dimension of $\mathfrak M^d_Y$.
\end{proof}

Now we are going to construct the relative Prym variety $$\PPPb\,\,=\,\,\Prym(\JJJ^0/H_0,\,\iota).$$
Consider the following three involutions on $\JJJ^0$:
\begin{enumerate}
\item $\iota^*$:\, Where $\iota\,:\,Y\, \longrightarrow\, Y$ is, as before, the involution that
permutes the two components of $Y$.

\item $\epsilon$:\, The inversion that sends a sheaf of the form $j_{u*}\LLL$ 
to $j_{u*}(\LLL^{-1})$, for each line bundle $\LLL$ of bidegree $(0,\,0)$ on the curve $\Gamma_u$ with $u\,\in\, 
H_0$.

\item The composition of the above maps $\tau\,=\,\iota^*\circ \epsilon\,=\,\epsilon\circ \iota^*$.
\end{enumerate}
The fact that $\epsilon$ 
is a well-defined automorphism of the moduli space is proved as done in \cite{MT}: One observes that $\epsilon$ 
can be given by another formula: $$\epsilon\,\,:\,\,\FFF\,\,\longmapsto\,\, \EXT^1_{\OOO_Y}(\FFF,\, \HHH^{-1}),\ \
\, \text{where\ }\ \, \HHH\,\simeq\,\OOO_Y(\Gamma_t),$$ and this definition behaves well in families, meaning
it commutes with the operation of
base change. Let $\JJJ^{0\tau}\,\subset\, \JJJ^0$ denote the fixed point locus for
the involution $\tau$. It evidently projects to the locus of symmetric curves $\Gamma_u$ in $Y$; these are the
curves of the form $C\cup \rho(C)$, where $C$ runs over the open part 
$V$ of the Hilbert scheme of curves $\Hilb_S$ in $S$ parametrizing the smooth curves transversal to $D$ 
which lie in the given semiample integer cohomology class. We thus have a natural projection $\JJJ^{0\tau}
\, \longrightarrow\, V$. Now define $\PPPb$ as the connected component of $\JJJ^{0\tau}$ containing
the zero section of this projection $\JJJ^0\, \longrightarrow\, V$. 
It comes endowed with a natural map $\mu\,:\,\PPPb\, \longrightarrow\, V$ that sends each sheaf
$\FFF\,\in\, \PPPb$ to the curve $C\,\in\, V$ such that $\Supp\FFF\,=\,C\cup\rho(C)$.

To describe the sheaves in $\PPPb$, introduce some notation for gluing line bundles on $C$ and $C'\,=\,\rho(C)$.
Let $\MMM$ (respectively, $\MMM'$) be a line bundle on $C$ (respectively, $C'$). Let $\Delta
\,=\,C\cap C'\,=\,\{x_1,\,\cdots,\,x_\delta\}$, and fix some generators
$\boldsymbol e$ and $\boldsymbol e'$ of $\MMM\big\vert_{\Delta}$ and $\MMM'\big\vert_{\Delta}$ respectively.
One can think of $\boldsymbol e$ and $\boldsymbol e'$ as collections $\boldsymbol e
\,=\,(e_1,\,\cdots ,\,e_\delta)$ and $\boldsymbol e'\,=\,(e'_1,\,\cdots ,\, e'_\delta)$, where $e_i$
(respectively, $e'_i$) is a generator for the fiber $\MMM\big\vert_{x_i}$ (respectively, $\MMM'\big\vert_{x_i}$)
at $x_i$, $i\,=\,1,\,\cdots,\,\delta$. Then for any
$$\boldsymbol\lambda\,\,=\,\,(\lambda_1,\,\cdots,\,\lambda_\delta)\,\,\in\,\,(\CC^*)^\delta ,$$
we can describe the result
of gluing of $\MMM$ and $\MMM'$ --- at the points of $\Delta$ according to the gluing data
$(\boldsymbol e,\,\boldsymbol e',\,\boldsymbol \lambda)$ --- as the kernel sheaf in the following
short exact sequence:
$$
0\,\longrightarrow\, \MMM\#_{(\boldsymbol e,\boldsymbol e',\boldsymbol \lambda)}{\MMM'}
\,\longrightarrow\, \MMM\oplus\MMM' \,\xrightarrow{\,\,\,\kappa\,\,}\,\bigoplus_{i=1}^\delta\CC(x_i)
\,\longrightarrow\, 0,
$$
where $\kappa\,=\,\bigoplus_{i=1}^\delta\kappa_i$; for each $i\, \in\,
\{1,\, cdots,\, \delta\}$ the above map $\kappa_i$ is the composition of maps
$$
\MMM\oplus\MMM'\,\xrightarrow{\,\,\,\ev_{x_i}\,\,}\,
\MMM\big\vert_{x_i}\oplus\MMM'\big\vert_{x_i}\,\xrightarrow{\,\,\,(e^*_i,\,\lambda_i{e'_i}^*)\,\,} \,\CC(x_i),
$$
and $e^*_i$ (respectively, ${e'_i}^*$) denotes the linear form on $\MMM\big\vert_{x_i}$
(respectively, $\MMM'\big\vert_{x_i}$) dual to $e_i$ (respectively, ${e'_i}$). With this notation,
for every $\nu\,\in\,\CC^*$, we have
\begin{equation}\label{a1}
\MMM\#_{(\boldsymbol e,\boldsymbol e',\nu\boldsymbol \lambda)}{\MMM'}
\,\simeq\,\MMM\#_{(\boldsymbol e,\boldsymbol e',\boldsymbol \lambda)}{\MMM'},\ \,
\big(\MMM\#_{(\boldsymbol e,\boldsymbol e',\boldsymbol \lambda)}{\MMM'}\big)^{-1}\, \simeq
\, \MMM^{-1}\#_{(\boldsymbol e^*,\boldsymbol {e'}^*,\boldsymbol \lambda^{-1})}{\MMM'}^{-1},
\end{equation}
where $\nu\boldsymbol \lambda\,=\,(\nu\lambda_1,\,\cdots,\,\nu\lambda_\delta)$,
$\boldsymbol e^*\,=\,(e^*_1,\,\cdots ,\, e^*_\delta)$
and $\boldsymbol \lambda^{-1}\,=\,(\lambda_1^{-1},\,\cdots,\,\lambda_\delta^{-1})$. The first relation
in \eqref{a1} shows that $\boldsymbol\lambda$ can be considered modulo the simultaneous
multiplication by $\CC^*$ on the factors, in other words,
$\boldsymbol\lambda$ is an element of $(\CC^*)^\delta/\CC^*\,\simeq\,(\CC^*)^{\delta-1}$.

\begin{theorem} \label{pr-th}
Let $\PPPb\,=\,\Prym(\JJJ^0/H_0,\,\iota)$ be the relative Prym variety constructed above.
The following statements hold:
\begin{enumerate}
\item The fiber $\PPPb_v\,=\,\mu^{-1}(v)$ over any point $v\,\in\, V$ representing a curve $C\,=\,C_v$
is described as follows. Let $C'\,=\,\rho(C)$,\, $\rho_C\,=\,\rho\big\vert_C\,:\,C\,\isoto\, C'$, and associate
to any invertible sheaf $\MMM$ on $C$ the sheaf $^\rho\!\MMM\,:=\,(\rho_C^{-1})^*\MMM$ on $C'$. Then
the fibers $\MMM\big\vert_{x_i}$ and $^\rho\!\MMM\big\vert_{x_i}$ are naturally identified, so the choice of a basis
$\{e_i\}_{i=1}^\delta$ of $\MMM\big\vert_{x_i}$ determines a basis of $^\rho\!\MMM\big\vert_{x_i}$, which is denoted
by the same symbol $\{e_i\}_{i=1}^\delta$. With this notation, 
$$
\PPPb_v\,\,=\,\,\left\{\left[\MMM\#_{(\boldsymbol e,\boldsymbol e^*,
\boldsymbol \lambda)}{\big(^\rho\!\MMM^{-1}\big)}\right]
\ \big\vert \ \boldsymbol e\ \mbox{is a generator of}\ \MMM\big\vert_{\Delta},\
\boldsymbol\lambda\,\in\, (\CC^*)^\delta/\CC^*\right\}.
$$
It is a connected abelian subscheme of $J^0(C\cup C')$ of dimension $g+\delta-1$ that fits in the
short exact sequence
\begin{equation}\label{extP}
1\,\longrightarrow\, (\CC^*)^\delta/\CC^*\,\xrightarrow{\,\,\,i_C\,\,}\, \PPPb_v\,\xrightarrow{\,\,\,r_C\,\,}
\, J^0(C)\,\longrightarrow\, 0,
\end{equation}
in which the maps are given by
$$
\boldsymbol \lambda\,\mapstoo{i_C}\, \OOO_C\#_{(\boldsymbol 1,\boldsymbol 1,
\boldsymbol \lambda)}\OOO_{C'},\ \ \ 
\LLL\, \mapstoo{r_C} \,\LLL\big\vert_{C},
$$
where $\boldsymbol 1\,=\,(1,\,\cdots,\,1)$ ($\delta$ times).

\item $\PPPb\,=\,\Prym(\JJJ^0/H_0,\,\iota)$ is a smooth symplectic variety of dimension $2(g+\delta-1)$, whose
symplectic structure $\sigma_{\PPPb}$ is the restriction of $\sigma$, which is the symplectic structure on
$\JJJ^0\,=\,\mathfrak M^0_Y$. Moreover, the natural map
$\mu\,:\,\PPPb\,\longrightarrow\, V$ is a Lagrangian fibration.

\item Let $\CCC/V$ denote the universal family of curves on $S$ parametrized by $V$. Then the relative 
Jacobian $\JJJ(\CCC/V)$ is naturally identified with an open part of the moduli space $\mathfrak M_S$ of simple 
sheaves on $S$, and hence $\mathfrak M_S$ carries the Poisson structure $\PPP$ defined by \eqref{poi}. There
is a natural projection $\beta\,:\,\PPPb\,\longrightarrow\, \JJJ(\CCC/V)$ which sends a line bundle $\LLL$ on
a curve $C_v\cup C'_v$ to its restriction $\LLL\big\vert_{C_v}$. This projection is a principal $(\CC^*)^{\delta-
1}$--bundle, and the projection is also a map 
of Poisson manifolds, representing $(\JJJ(\CCC/V),\,\PPP)$ as a reduction of the nondegenerate Poisson 
structure $(\PPPb,\,\sigma_{\PPPb}^{-1})$.
\end{enumerate}
\end{theorem}

\begin{proof}
\textbf{(1).}\, The exact sequence \eqref{genJ} expresses the fact that every sheaf $\LLL$ from $\JJJ^0$ can 
be obtained as a result of gluing $(\LLL\big\vert_{C_u})\#_{(\boldsymbol e,\boldsymbol e',\boldsymbol 
\lambda)}(\LLL\big\vert_{C'_u})$, and the action of $(\CC^*)^\delta$ on $\JJJ^0$ by dilations of $\boldsymbol 
\lambda$ has the diagonal scalar dilations $\CC^*\, \subset\, (\CC^*)^\delta$ as the kernel. In view of the 
fact that the condition that $\LLL\,\in\,\JJJ^{0\tau}$ is equivalent to the conditions
that $\rho(C_u)\,=\,C'_u$ and $\LLL\big\vert_{C_u}\,\simeq\,
(\rho^*\LLL\big\vert_{C'_u})^{-1}$, the restriction of \eqref{genJ} to the fixed point locus of 
$\tau$ yields \eqref{extP}.

\textbf{(2).}\, One easily verifies that both the involutions $\iota^*$ and $\epsilon$ are
antisymplectic, meaning $\epsilon^*(\sigma) \,=\, (\iota^*)^*(\sigma)\,=\,-\sigma$, and hence
their composition $\tau$ is actually symplectic.

It is a standard fact that every connected component of the fixed locus of a regular involution on
a smooth variety is a smooth subvariety whose tangent spaces are the $+1$--eigenspaces of the
involution acting on the tangent spaces at the fixed points. It is also a straightforward exercise
to see that the $+1$--eigenspace of a symplectic involution on a symplectic vector space
is a symplectic subspace. Thus $\PPPb$ is a smooth symplectic subvariety of $\JJJ^0$.

The computation of the dimension of $\PPPb$ follows from part (1). The fibers of $\mu$
are contained in those of the projection $\JJJ^0\,\longrightarrow\, H_0$, and this
projection $\JJJ^0\,\longrightarrow\, H_0$ is Lagrangian. Therefore the fibers of $\mu$ are
indeed totally isotropic. As $\dim \PPPb\,=\,2\dim V$, the fibers of $\mu$ must be Lagrangian.

\textbf{(3).}\, Let $\FFF\,=\,j_{\Gamma*}\LLL$ --- where $\LLL$ is a line bundle on $\Gamma$ --- be a
point of $\PPPb$. Denote $\LLL_C\,:=\,\LLL\big\vert_C$ and $\LLL_{C'}\,:=\,\LLL\big\vert_{C'}\,\simeq\, (\rho^{-1})^*
\LLL_C^{-1}$. To simplify the notation, we will write simply $\LLL$ in place of $\FFF$ and omit
$j_{\Gamma*}$ when applied to a sheaf supported on (a part of) $\Gamma$.

We will represent the elements of $\Ext^1_Y(\LLL,\,\LLL)$ as classes in the hypercohomology group
$\HH^1(K^\bdot)$, where $K^\bdot\,=\,K_Y^\bdot$ is the simple complex of locally free sheaves
associated to the double complex $\HOM_{\OOO_Y}(L^\bdot,\,L^\bdot)$, and $L^\bdot$
is a locally free resolution of $\LLL$ as an $\OOO_Y$--module, which can be chosen to be of
length 1. Since $\Gamma$ is a Cartier curve and the local equations of $C$ at the singular points
of $\Gamma$ are just the restrictions of local equations of $\Gamma$ to $S$, the restriction $L_S^\bdot
\,=\,L^\bdot\otimes_{\OOO_Y}\OOO_S$ is a locally free resolution of $\LLL_C$. Thus the restriction map $\beta
\,:\,K^\bdot\, \longrightarrow\,K_S^\bdot$ of simple complexes --- associated to the double complexes
$\HOM_{\OOO_Y}(L^\bdot,\,L^\bdot)$ and $\HOM_{\OOO_S}(L_S^\bdot,\,L_S^\bdot)$ ---
provides the following map of the hypercohomologies:
$$
\beta_*\,:\,\Ext^1_Y(\LLL,\,\LLL)\,=\,\HH^1(K^\bdot)\, \longrightarrow\, \Ext^1_S(\LLL_C,\,\LLL_C)
\,=\,\HH^1(K_S^\bdot),
$$
which is nothing else but $\ud_{[\LLL]}r_C$, the differential of the map
$r_C$ (see \eqref{extP}) at the point $[\LLL]\,\in\,\PPPb$. The local-to-global spectral sequence gives
the following commutative diagram with exact rows, in which we simplify the Ext-sheaves by the formulas like
$\EXT^i(\LLL,\,\MMM)\,=\,\LLL^{-1}\otimes\MMM$ for $i\,=\,0$ and $\LLL^{-1}\otimes\MMM
\otimes\NNN_{\Gamma/Y}$ for $i\,=\,1$:
$$
\xymatrix{
0 \ar[r] & H^1(\OOO_\Gamma) \ar[r] \ar [d] &
 \Ext^1_Y(\LLL,\,\LLL)   \ar[r]  \ar^{\beta_*}[d] &
                                                                  H^0(\NNN_{\Gamma/Y}) \ar[r] \ar [d] & 0 \\
0 \ar[r] & H^1(\OOO_C) \ar[r]  
      & \Ext^1_S(\LLL_C,\,\LLL_C)  \ar[r]   
                                                                  & H^0(\NNN_{C/S}) \ar[r]  
                                                                  & 0 
}
$$
Note that $\OOO_\Gamma$ has a natural $\OOO_Y$-resolution of the form $\OOO_Y(-\Gamma)\, \longrightarrow\,
\OOO_Y$, whose restriction to $S$ is the natural resolution $\OOO_S(-C)\, \longrightarrow\,\OOO_S$
of $\OOO_C$. So we have
$\NNN_{\Gamma/Y}\big\vert_C\,\simeq\, \NNN_{C/S}$.

To obtain a similar diagram containing the codifferential $\ud\,\check{}_{[\LLL]}r_C\,=\,
\beta_*\check{}$, we dualize and use the Serre duality, keeping in mind the facts that
$\omega_Y$ is trivial and $\omega_S\,\simeq\,\OOO_S(-D)$:
\begin{equation}\label{el2}
\begin{gathered}\xymatrix{
0 \ar[r] & H^1(\OOO_C(-\Delta)) \ar[r] \ar [d] &
      \Ext^1_S(\LLL_C,\,\LLL_C(-D))   \ar[r]  \ar^{\beta_*\check{}}[d] &
                                                                  H^0(\omega_C) \ar[r] \ar [d] & 0 \\
0 \ar[r] & H^1(\OOO_\Gamma) \ar[r]
      & \Ext^1_Y(\LLL,\,\LLL) \ar[r]
                                   & H^0(\omega_\Gamma) \ar[r]
                                                                  & 0
}
\end{gathered}
\end{equation}

The verification that $\beta_*\check{}$ in \eqref{el2} respects the Poisson structures
is straightforward once one invokes the following observation: When 
we represent one of the relevant $\Ext^1$--groups by an extension of $H^1(\EXT^0(\cdot,\,\cdot))$ by
$H^0(\EXT^1(\cdot,\,\cdot)$, the $H^1$ 
part is totally isotropic and the Poisson pairing is completely determined by the pairing between
$H^1(\EXT^0(\cdot,\,\cdot))$ and $H^0(\EXT^1(\cdot,\,\cdot))$. When working on $S$, the result should
be contracted with the Poisson structure of $S$, that is multiplied by $\gamma\,\in\, H^0(S,\,\OOO_S(D))$.
The vertical maps in \eqref{el2} obviously respect the natural pairings between the $H^1$ and $H^0$
spaces in each row. This completes the proof of the theorem.
\end{proof}

\begin{remark}\label{any-d}
The above construction provides the symplectic manifold $\PPPb/V$ whose Poisson reduction is the abelian 
scheme $\JJJ(\CCC/V)$ parametrizing the line bundles of
degree zero on the fibers of $\CCC/V$. We can slightly modify 
this construction to cover the case of torsors $\JJJ^d(\CCC/V)$ for any $d\in\ZZ$. To this end, we generalize 
the definition of the relative Prym variety as follows. Let $\JJJ^{(d,-d)}(\bGamma/H_0)$ be the relative 
Picard scheme parametrizing the line bundles of bidegree $(d,\,-d)$ on the reducible curves $\Gamma_u\,=\,C_u\cup 
C'_u$, where $u\,\in\, H_0$, $C_u\,\subset\, S$
and $C'_u\,\subset\, S'$. We keep unchanged the definitions of the involutions 
$\epsilon$, as well as that of $\iota$, but note that this time $\iota^*$ sends 
$\JJJ^{(d,-d)}(\bGamma/H_0)$ to another component of the relative Picard scheme, namely 
$\JJJ^{(-d,d)}(\bGamma/H_0)$. Nevertheless the composition of maps
$$\tau\,=\,\iota^*\circ\epsilon\,=\,\epsilon\circ \iota^*$$
is again a symplectic involution of $\JJJ^{(d,-d)}(\bGamma/H_0)$. It is clear that the line bundles 
$\LLL_{C_v}\#{\rho^{-1}}^*\LLL_{C_v}^{-1}$ are $\tau$--invariant for any gluing factors $\lambda_i$ at the 
points $x_i\,\in\, C_v\cap C'_v$, where $v\,\in \,V$ and $C'_v\,=\,\rho(C_v)$, and these line bundles fill
the relative Prym variety $\PPPb^d$ of degree $d$, which is a torsor under $\PPPb\,=\,\PPPb^0$. It is
symplectic and has $\JJJ^d(\CCC/V)$ as its Poisson reduction by the same argument as above, provided
the sheaves from $\JJJ^{(d,-d)}(\bGamma/H_0)$ are actually $h$-stable. The stability condition is
equivalent to the inequality $2|d|\,<\, \delta$; for such $d$ our argument shows
that $\PPPb$ is quasi-projective and it has an algebraic symplectic structure
satisfying the condition that the Poisson structure on $\JJJ^d$ is its reduction. For $2|d|
\,\geq\, \delta$, the line bundles $\LLL$ on $C\cup C'$ of bi-degree $(d,\,-d)$ are no
longer $h$--stable. If, say $d\,\geq\, \frac\delta2$, then $\LLL$ is destabilized by the exact triple
$$
0\,\longrightarrow\, \LLL_C(-\Delta)\,\longrightarrow\, \LLL\,\longrightarrow\,\LLL_{C'}\,\longrightarrow\,0,
$$ since $h$ has equal degrees on $C,\,C'$ and $$\deg \LLL_C(-\Delta)\,=\,d-\delta
\,\geq\,\deg \LLL_{C'}\,=\,-d.$$ But these sheaves are still simple, so we obtain the complex analytic
analog of the desired statement by working with the relative analytic moduli space of simple sheaves: assuming
that $2|d|\,\geq\, \delta$, the relative Prym variety $\PPPb^d$ is a complex manifold with a holomorphic symplectic structure, and the Poisson manifold $\JJJ^d$ is its Poisson reduction. 
\end{remark}

\section{Moduli of 1-dimensional framed sheaves}\label{framed}

Start with the same initial data:\, The Poisson surface $(S,\,\gamma)$ and the anticanonical curve
$D\,=\,{\rm Div}_S(\gamma)$, which we will assume to be smooth and connected. 

\begin{defn}\label{fr-def}
Let $\theta\,=\,\theta_D$ be an $\OOO_D$--module, and let $j_D\,:\,D\,\into\, S$ be
the natural embedding. A $\theta$-framed sheaf on $S$ is a pair $(\FFF,\,\phi)$, where
$\FFF$ is an $\OOO_S$-module and $\phi\,:\,\FFF\, \longrightarrow\, j_{D*}\theta$ is a
homomorphism of $\OOO_D$-modules inducing an isomorphism $\phi_D\,:\,\FFF\big\vert_{D}\,\isoto\,\theta$ when
it is restricted to $D$. Two framed sheaves $(\FFF,\,\phi)$ and $(\GGG,\,\psi)$ are isomorphic if there is
an $\OOO_S$-isomorphism $f\,:\,\FFF\,\isoto \,\GGG$ such that $\phi\,=\,\psi\circ f$.
\end{defn}

In the sequel we will consider the situation when $\theta$ is a line bundle on $D$ and the sheaves $\FFF$ are 
pure 1-dimensional. Such a sheaf $\FFF$ can only be $\theta$--framed if the support of $\FFF$ contains $D$ as an 
irreducible component and $\FFF\big\vert_D\,\simeq\,\theta$. As in Section 1, we will not distinguish in writing
between the 
sheaves supported on closed subschemes of $S$ and the sheaves on $S$ obtained as their direct images by the 
inclusion map of the support into $S$. Thus we will simply write $\LLL$ instead of $j_{Z*}(\LLL)$ when 
considering a line bundle on a curve $Z\,\subset\, S$ as an $\OOO_S$--module, and also write $\phi\,:\,\FFF\, 
\longrightarrow\, \theta$ instead of $\phi\,:\,\FFF\, \longrightarrow\, j_{D*}\theta$ to denote a framing of 
$\FFF$.

Consider the universal family $\CCC/V$ of smooth connected curves $C$ over the open part 
$V$ of the Hilbert scheme of curves $\Hilb_S$ in $S$, singled out by the conditions that $C\cap D
\,\neq\,\varnothing$ and $C$ intersects $D$ transversely while lying in a fixed semiample
integral cohomology class of $S$. For each curve $C$ from this family, denote $\widetilde C\,=\,C\cup D$; the
curves $\widetilde C$ fit into the family $\widetilde\CCC/V$ of nodal curves over $V$ with fixed component
$D$. The classifying morphism for this family embeds $V$ as a locally closed subscheme in another component
of $\Hilb_S$; let us denote by $W_0$ the image of this embedding, so that $\widetilde\CCC\,\longrightarrow\,
W_0\,\simeq\, V$
is the restriction of the universal curve over $\Hilb_S$ to $W_0$.
Let now $W$ be a sufficiently small open neighborhood of $W_0$ in $\Hilb_S$ parametrizing deformations
of the curves from $\widetilde \CCC$ with at worst nodal singularities which do not acquire irreducible
components other than $D$; denote by $\ZZZ\,\longrightarrow\, W$ the universal curve over it. We have
$$\widetilde\CCC\,=\,\ZZZ\times_{W}W_0\,=\,\ZZZ\big\vert_{W_0};$$
the curve $Z_t$ is nodal and it is irreducible for every $t\,\in\, W\setminus W_0$.

For two integers $d_0,\,d_1\,\in\,\ZZ$, we will consider the following Picard schemes:
$$J^{d_0}_D\,=\,\Pic^{d_0}(D),\ \,\, \JJJ_{W}^{d}\,=\,\mathbf{Pic}^{d}(\ZZZ/W),\ \,\,
\JJJ_{V}^{d_1}\,=\,\mathbf{Pic}^{d_1}(\CCC/V),$$ where $d\,=\,d_0+d_1$ and the sheaves represented
by the points of the relative Picard schemes are identified with the corresponding simple sheaves of
$\OOO_S$--modules. Fixing $d_0,\,d_1$ and $d\,=\,d_0+d_1$ once and for all, we will omit the superscripts,
by simply writing $J_{D},\,\JJJ_{W},\, \JJJ_{V}$. We will also consider the 
restriction $\JJJ_{W_0}$ of $\JJJ_W$ over $W_0\,\simeq\, V$ with its two natural projections:
\begin{equation}\label{e-1}
J_D\,\xleftarrow{\,\,\pi_D\,\,\,}\,\JJJ_{W_0}\,=\,\JJJ_W\times_{W}W_0\,\xrightarrow{\,\,\,\pi_V\,\,}\,\JJJ_V.
\end{equation}
The support of a sheaf $\LLL$ in $\JJJ_{W_0}$ is a reducible curve $\widetilde C\,=\,C_v\cup D$ for some $v
\,\in\, V$, and
$$\pi_D([\LLL])\,=\,[\LLL\big\vert_{D}],\ \ \, \pi_V([\LLL])\,=\,[\LLL\big\vert_{C_v}].$$
By \cite{HM}, $\JJJ_W$ and $\JJJ_{V}$ are Poisson manifolds with the respective Poisson structures
$\PPP_W$ and $\PPP_V$ defined by pairings of type \eqref{poi}. For any $\theta\,\in\, J^{d_0}_D$ the framed
sheaves fill the closed subvariety 
$$
\JJJ_{W}^{\theta}\ =\ \JJJ_{W_0}^{\theta}\ =\ \pi_D^{-1}(\theta),
$$
which we will denote by $\JJJ^\theta$ for notational convenience.
The deformation theory of framed sheaves was studied in \cite{L}, \cite{HL-1}, \cite{HL-2} and \cite{BM}, but
no one among these papers covers the case of torsion framed sheaves. For two dimensional torsion-free
sheaves $\FFF$ framed on $D$ by an isomorphism $\FFF\big\vert_D\,\simeq\, F_D$ for a fixed vector bundle $F_D$ on
$D$, the tangent space $T_{[\FFF]}\mathfrak M_S^{\mathrm{fr},F_D}$ to the moduli space is identified with
$\Ext^1_S(\FFF,\,\FFF(-D))$. It can equally be represented as the hyper-Ext group
$\hExt^1_S(\FFF,\,\FFF\to\FFF\big\vert_D)$. It turns out that the latter formula still holds for 1-dimensional
sheaves with a framing understood in the sense of Definition \ref{fr-def}. 

\begin{theorem}\label{fr-th}
Using the above notation, assume that $\LLL\,\in\,\JJJ_{W_0}^{\theta}$. 
\begin{enumerate}
\item
The tangent space of the framed moduli space at $\LLL$ admits the following natural identifications:
$$
T_{[\LLL]}\JJJ_{W_0}^{\theta}\ =\ \hExt^1_S(\LLL,\,\,\LLL\to\LLL\big\vert_D)\ =\ \Ext^1(\LLL,\,\,\LLL_C(-D)).
$$
\item
This tangent space is naturally self-dual via the skew-symmetric perfect pairing:
\begin{equation}\label{e0}
 \PPP^\theta\, \,:\,\, 
\Ext^1(\LLL,\,\LLL_C(-D))\times \Ext^1(\LLL,\,\LLL_C(-D))
\end{equation}
$$
\ \ \ \ \ \ \ \ \ \ \ \ \ \ \ \ \ \ \ \ \ \ \ \ \ \ \ \ \xrightarrow{\,\,\,(\xi,\eta)\mapsto \xi\cup({i_C}_*\eta)\,\,}\,\, \Ext^2(\LLL,\,\LLL_C(-D))\,\,=\,\,\CC,
$$
where $i_C\,:\,\LLL_C(-D)\,\into\, \LLL$ is the natural inclusion map, and the last isomorphism with $\CC$
is given by the composition of maps:
$$
\Ext^2(\LLL,\,\LLL_C(-D))\,=\,H^1(\EXT^1(\LLL,\,\LLL_C(-D)))\,=\,H^1(\omega_C)\,=\,\CC.
$$

\item The pairing $\PPP^\theta$ in \eqref{e0} defines a Poisson structure on $\JJJ^\theta$ making it
into a symplectic leaf for the Poisson structure $\PPP_W$ on $\JJJ_W$.

\item The restriction map $\pi_V$ in \eqref{e-1} defined by $\LLL'\,\longmapsto\, \LLL'\big\vert_C$ is a Poisson reduction.
\end{enumerate}
\end{theorem}

\begin{proof}
\textbf{(1) and (2).}\,\, For a sheaf $\LLL\,\in\, \JJJ_{W_0}$, we have $\Supp\LLL\,=\,\widetilde C\,=\,C\cup D$.
The tangent space to deformations of $\widetilde C$ is $H^0(\NNN_{\widetilde C/S})$, while local
deformations (smoothings) of the nodes are controlled by the Schlessinger's $T^1$--sheaf $$\TTT^1_{\widetilde C}
\ =\ \EXT^1_{\OOO_{\widetilde C}}(\Omega^1_{\widetilde C},\,\OOO_{\widetilde C})\ \simeq \
\bigoplus_{i=1}^\delta\CC_{x_i}.$$ There is a natural map of sheaves of local deformations $\ell\,:\,
\NNN_{\widetilde C/S}\, \longrightarrow\, \TTT^1_{\widetilde C}$; see \cite[Theorem 6.2]{A}.
The subspace tangent to equisingular deformations preserving the $\delta$ nodes is
$H^0(\widetilde \NNN_{\widetilde C/S})$, where
$$\widetilde \NNN_{\widetilde C/S}\,\,:=\,\,\ker \ell\,\,=\,\,\III_{\Delta,\tilde C}\cdot\NNN_{\tilde C/S};$$
this is at the same time the tangent subspace to deformations preserving the decomposition into
two irreducible components $C'\cup D'$, of which the first one
is a deformation of $C$ and the second one is a deformation of $D$. This manifests itself in the decomposition
$\widetilde \NNN_{\tilde C/S}\,\simeq\, \NNN_{C/S}\oplus \NNN_{D/S}$.
The infinitesimal deformations keeping $D$ as a fixed component of the deformed curve
$\widetilde C$ project to zero in $H^0(\NNN_{D/S})$, meaning it belong to $H^0(\NNN_{C/S})$. Thus, to single out 
$T_{[\LLL]}\JJJ_{W_0}$ in $\Ext^1(\LLL,\,\LLL)$, we start with the short exact sequence
\begin{equation}\label{Def-rC}
0\, \longrightarrow\, H^1(\OOO_{\tilde C})\, \longrightarrow\, \Ext^1(\LLL,\,\LLL)\,\xrightarrow{\,\,\,
p_{\widetilde C}\,\,}\, H^0(\NNN_{\widetilde C/S})\, \longrightarrow\, 0
\end{equation}
and consider the composition of homomorphisms $$\Ext^1(\LLL,\,\LLL)\, \longrightarrow\, H^0(\NNN_{\widetilde C/S})
\, \longrightarrow\, H^0(\NNN_{D/S}(\Delta)),$$
in which the second map comes from the exact sequence $$0\, \longrightarrow\,\NNN_{C/S}\, \longrightarrow\,
\NNN_{\tilde C/S}\, \longrightarrow\,
 \NNN_{D/S}(\Delta)\, \longrightarrow\, 0.$$
Let
$$\chi\,\, :\,\, \Ext^1(\LLL,\,\LLL)\,\, \longrightarrow\,\,H^0(\NNN_{D/S}(\Delta))$$
be this composition of homomorphisms.
Then $$T_{[\LLL]}\JJJ_{W_0}\,\,=\,\,p_{\tilde C}^{-1}\big(H^0(\NNN_{C/S})\big)\,\,=\,\,\ker \chi$$
fits in the exact sequence
 \begin{equation}\label{TJW0}
 0\, \longrightarrow\, H^1(\OOO_{\tilde C})\, \longrightarrow\, T_{[\LLL]}\JJJ_{W_0}
\, \longrightarrow\, H^0(\NNN_{C/S})\, \longrightarrow\, 0.
 \end{equation}
 We see that $\dim T_{[\LLL]}\JJJ_{W_0}\,=\,\dim \JJJ_{W_0}\,=\,C^2+\delta+1$, so $\JJJ_{W_0}$ is a smooth
subvariety of $\JJJ_W$ of dimension $C^2+\delta+1$.
 
As $\pi_D$ is a fibration over the elliptic curve $\JJJ_D$, if we fix $\theta\,\simeq\,\LLL\big\vert_D$, then
we obtain the fiber $\pi_D^{-1}(\theta)\,=\,\JJJ^\theta$ of dimension $C^2+\delta$ whose tangent space
fits in the following commutative diagram:
$$
\xymatrix{
&0\ar [d] &0\ar [d] &&\\
0 \ar[r] & H^1(\OOO_C(-\Delta))\ar[r] \ar [d] &
        T_{[\LLL]}\JJJ^\theta\ar[r] \ar [d] &
             H^0(\NNN_{C/S}) \ar[r]  \ar @{=}[d] & 0 \\
0 \ar[r] & H^1(\OOO_{\tilde C}) \ar[r] \ar [d] &    
        T_{[\LLL]}\JJJ_{W_0} \ar[r] \ar [d] &  H^0(\NNN_{C/S}) \ar[r] & 0\\
        & H^1(\OOO_D) \ar @{=}[r] \ar [d] & H^1(\OOO_D) \ar [d] &&\\
&0&0&&
}
$$
{}From the commutative diagram
\begin{equation}\label{rC*}
\begin{gathered}\xymatrix{
      0 \ar[r] & H^1(\OOO_C(-\Delta)) \ar[r] \ar[d]  &
      \Ext^1(\LLL_C,\,\LLL)  \ar[r]\ar[d]^{{r_C}^*}    &
                                                                  H^0(\NNN_{C/S}) \ar[r] \ar[d]  & 0\\
0 \ar[r] & H^1(\OOO_{\tilde C}) \ar[r]  &
      \Ext^1(\LLL,\LLL)  \ar[r]   &    H^0(\NNN_{\tilde C/S}) \ar[r]  & 0
}
\end{gathered}
\end{equation}
where ${r_C}^*$ is the pullback map associated to the morphism of sheaves $r_C\,:\,\LLL\, \longrightarrow\,\LLL_C$,
we deduce the identification
$$
T_{[\LLL]}\JJJ^\theta\,=\,\im {r_C}_*\,=\, \Ext^1(\LLL_C,\,\LLL)
\,\subset\, T_{[\LLL]}\JJJ_W\,=\, \Ext^1(\LLL,\,\LLL).
$$
Dualizing, we obtain the cotangent space fitting in the dual of the top exact
sequence in \eqref{rC*}:
\begin{equation}\label{z1}
0 \, \longrightarrow\, H^1(\OOO_C(-\Delta))\,  \longrightarrow\, {T\:\check{}\!_{[\LLL]}}\JJJ^\theta
\,=\,\Ext^1(\LLL,\,\LLL_C(-D))\,  \longrightarrow\, H^0(\NNN_{C/S}) \,  \longrightarrow\, 0.
\end{equation}
The local-to-global spectral sequence degenerates at the $E_2$--terms and respects the Yoneda product
modulo signs. More precisely, the cup-product $$E_2^{p_1,q_1}\times E_2^{p_2,q_2}\ \longrightarrow\
E_2^{p_1+p_2,q_1+q_2}$$ differs from the Yoneda product induced on the {$(p_i,\,q_i)$--graded pieces of
$\Ext^{p_i+q_i}(\bdot,\,\bdot)$} by the sign $(-1)^{p_1q_2}$. When working with $1$-dimensional sheaves, the
$E^{1,0}$-part of $\Ext^{1}(\bdot,\,\bdot)$ is Yoneda-isotropic (we have used this property in the proof
of Theorem \ref{pr-th}), so the Yoneda pairing on $\Ext^1$ is determined by the cup-product $$E_2^{1,0}\times
E_2^{0,1}\ \longrightarrow\ E_2^{1,1}.$$ We can choose any section of the surjection in
\eqref{z1} and use it to represent covectors from $T\,\check{}_{[\LLL]}\JJJ^\theta$ by the pairs
$(\xi,\,\eta)\,\in\, H^1(\OOO_C(-\Delta))\times H^0(\NNN_{C/S})$. Taking two such covectors $(\xi_1,\,\eta_1)$
and $(\xi_2,\,\eta_2)$, we obtain the following:
$$
\PPP^\theta((\xi_1,\,\eta_1),\,(\xi_2,\,\eta_2))\ =\ -\xi_1\cdot\eta_2+{\xi_2\cdot\eta_1},
$$
where the product ``$\xi\cdot\eta$'' stands for the natural pairing
\begin{equation}\label{dot}
H^1(\OOO_C(-\Delta))\times H^0(\NNN_{C/S})\, \longrightarrow\, H^1(\NNN_C(-\Delta))\,=\,H^1(\omega_C)\,=\,\CC.
\end{equation}

We thus see that $\PPP^\theta$ is skew-symmetric and non-degenerate. In particular, we have
a natural identification of the tangent and the cotangent space, so we obtain natural isomorphisms
$$
T_{[\LLL]}\JJJ^\theta\,=\ {T\:\check{}\!_{[\LLL]}}\JJJ^\theta\,=\
\Ext^1(\LLL_C,\,\LLL)\,=\ \Ext^1(\LLL,\,\LLL_C(-D)).
$$
This completes the proof of (1) {and} (2).

\textbf{Proof of (3).}\,\, The differential of the inclusion map $\JJJ^\theta\,\into\, \JJJ_W$ is the map ${r_C}^*$ in
the diagram \eqref{rC*}. Dualizing, we obtain
the following:
\begin{equation}\label{z2}
\begin{gathered}
\xymatrix@C=15pt{
0 \ar[r] & H^1(\OOO_{\tilde C}(-D)) \ar[r]\ar@{->>}[d]  &
      \Ext^1(\LLL(D),\,\LLL)  \ar[r] \ar@{->>}[d]^{{r_C}_*}  & H^0(\omega_{\widetilde C}) \ar[r] \ar@{->>}[d] & 0\\
      0 \ar[r] & H^1(\OOO_C(-\Delta)) \ar[r]   &
      \Ext^1(\LLL(D),\,\LLL_C)  \ar[r]    &
                                                                  H^0(\NNN_{C/S}) \ar[r]   & 0\\
}
\end{gathered}
\end{equation}
As above, we can split the surjections in the {rows} of
\eqref{z2} in a compatible way, so that the middle column is the direct sum of the first and the third
columns; such a splitting is not unique, however the result of computing of the pairing does not depend on
the choice of a splitting. We will now verify that the restriction of $\PPP_W$ to $\JJJ^\theta$ is equal, up
to a constant factor, to the bivector field $\PPP^\theta$ defined in part (2).

It suffices to see that
for any pair $(\xi,\,\eta)\,\in\, H^1(\OOO_C(-\Delta))\times H^0(\NNN_{C/S})$ which is the image of a pair
$(\widetilde\xi,\,\widetilde\eta)\,\in\, H^1(\OOO_{\widetilde C}(-D))\times H^0(\omega_{\tilde C})$, we have
{$\xi\cdot\eta\ =\ \widetilde\xi\circ\widetilde\eta$}, where the {dot} ``$\cdot$'' is the
pairing \eqref{dot} and ``$\circ$'' denotes the following one:
\begin{equation}\label{pp}
H^1(\OOO_{\widetilde C}(-D))\times H^0(\omega_{\tilde C})\, \longrightarrow\, H^1(\omega_{\tilde C}(-D))
\,\xrightarrow{\,\,\, \cdot\gamma\big\vert_{\tilde C}\,\,}\, H^1(\omega_{\widetilde C})\,=\,\CC.
\end{equation}

This can be verified by computing both pairings via products of \v Cech cocycles. We can cover $\widetilde C$
by two affine charts 
$\widetilde U_0\,=\,f^{-1}(\PP^{1}\setminus\infty)$ and $\widetilde U_1\,=\,f^{-1}(\PP^{1}\setminus0)$ for a
finite morphism $f\,:\,\widetilde C\,\longrightarrow\,\PP^1$ such that $f(\Delta)
\,\subset\, \PP^{1}\setminus\{0,\,\infty\}$. Then the intersections $U_i\,=\,\widetilde U_i\cap C$ form an
affine covering of $C$. Working on $\widetilde C$, we introduce sections $\widetilde e_i$ trivializing
$\OOO_S(-D)$ over $\widetilde U_i$, $i\,=\,0,\,1$. The triviality of $\OOO_S(-D)\big\vert_{\widetilde U_i}$ imposes some
restrictions on the choice of $f$; we can always specify such a choice. Say, if $\OOO_S(D)\big\vert_D$ is non-trivial,
we can choose $f$ as the map defined by a
generic pencil in $|D|$ containing $D$. Otherwise we define $f$ on $C$ as the map given by a generic pencil in
$|\Delta|$, containing $\Delta$, and extend it to $\tilde C$ by gluing it with an arbitrary non-constant map
$g\,:\,D\,\longrightarrow\,\PP^1$ subject to the condition that $g\big\vert_{\Delta}
\,=\,f\big\vert_{\Delta}$. In both the cases we obtain a map $f\,:\,\widetilde C\, \longrightarrow\,
\PP^1$ such that $c\,=\,f(\Delta)$ is a single point in $\PP^1\setminus\{0,\,\infty\}$. Over
$\widetilde U_i$, the section $\gamma\big\vert_{\widetilde C}$ is represented by some function
$\widetilde\gamma_i\,\in\, H^0(\widetilde U_i,\,\OOO_{\widetilde C})$, and $\widetilde \xi$ can be given by a function
$\widetilde\xi_{01}\,\in\, H^0(\widetilde U_0\cap \widetilde U_1,\,\OOO_S(-D)\big\vert_{\widetilde U_i})$ in the basis
$\tilde e_0$ of $\OOO_S(-D)\big\vert_{\widetilde U_0}$, understood as a 1-cochain of the
covering $(\widetilde U_i)$, or as a
1-cocycle --- which are same because our covering has no triple intersections --- with coefficients
in $\OOO_S(-D)\big\vert_{\widetilde C}$.

Similarly, we can represent $\xi$ by a 1-cochain $\xi_{01}\,\in\, H^0(U_0\cap U_1,\,\OOO_{C})$ viewed as a meromorphic section
of $\OOO_S(-D)\big\vert_{U_0}$ written in terms of the basis $e_0$, where
$e_i\,:=\,\widetilde e_i\big\vert_{U_i}$. Then we have the following:
$$
\widetilde\xi\circ\widetilde \nu\,=\,\sum_{x\in f^{-1}(0)}\Res_x(\widetilde\xi_{01}\widetilde\nu\,
\widetilde\gamma_0\tilde e_0),\ \ \ 
\xi\cdot\nu\,=\,\sum_{x\in f_C^{-1}(0)}\Res_x(\xi_{01}\nu),
$$
where $f_C\,=\,f\big\vert_C$. Since $\gamma$ vanishes on $D$, also $\widetilde\gamma_0$ is identically zero
over $D\cap \widetilde U_0$, it follows that only the points $x\,\in\, f_C^{-1}(0)$ contribute to the sum
of residues expressing $\widetilde\xi\circ\widetilde \nu$.
Furthermore, our hypotheses imply that
$\widetilde\xi_{01}\widetilde\nu\big\vert_{U_0}\,=\,\xi_{01}\nu$, and in view of our choice of
$f$, we know that $\widetilde\gamma_0\widetilde e_0$ takes the
same value
\begin{equation}\label{ec}
0 \ \neq\ c \ \in\ {\mathbb C}
\end{equation}
at all the points of
$f_C^{-1}(0)$, so the first sum of residues is $c$ times the second one. This proves that
$\PPP_W\big\vert_{\JJJ^\theta}\,=\,c\PPP^\theta$.

\textbf{Proof of (4).}\,\, The differential $\ud_{[\LLL]}\pi_V$ is the natural map
$${r_C}_*\ :\  \Ext^1(\LLL_C,\,\LLL)\ \longrightarrow \  \Ext^1(\LLL_C,\,\LLL_C)$$
given by the restriction homomorphism $\LLL\, \longrightarrow\, \LLL_C$. 
Upon dualizing it, we obtain the codifferential  $\ud\,\check{}_{[\LLL]}\pi_V$ as the middle
vertical arrow $r_C^*$ in the diagram
\begin{equation}\label{ec2}
\xymatrix{
0 \ar[r] & H^1(\OOO_C(-\Delta)) \ar[r] \ar @{=}[d] &
      \Ext^1(\LLL_C,\,\LLL_C(-\Delta))  \ar[r]  \ar@{^(->}^{{r_C^*}}[d] &
                                                                  H^0(\omega_C) \ar[r]  \ar@{^(->} [d]& 0 \\
0 \ar[r] & H^1(\OOO_C(-\Delta)) \ar[r] &
      \Ext^1(\LLL,\,\LLL_C(-\Delta))  \ar[r]   &
                                                                  H^0(\NNN_{C/S}) \ar[r]  & 0 \\
}
\end{equation}
We want to see that $\pi_V$ is a Poisson submersion, modulo scaling one of the Poisson structures by a constant
scalar factor. One proves that
$$\PPP_{V,[\LLL_C]}(\zeta_1,\,\zeta_2)\,=\,c\cdot \PPP^{\theta}_{[\LLL]}(r_C^*\zeta_1,\,r_C^*\zeta_2)$$
for any $\zeta_1,\,\zeta_2\,\in \, \Ext^1(\LLL_C,\,\LLL_C(-\Delta))$, where $c$ is the constant
in \eqref{ec}. This can be done in the same manner as above: one represents the $\zeta_i$ by pairs
$(\xi_i,\,\eta_i)\,\in \,H^1(\OOO_C(-\Delta))\times H^0(\omega_C)$, using any splitting of the
commutative diagram in \eqref{ec2} which turns its middle column into the direct sum of the first
column and the third column, while the relevant cup products are computed explicitly in terms of  \v Cech cocycles.
\end{proof}

\begin{remark}
Part (3) of Theorem \ref{fr-th} is a particular case of \cite[Theorem 3.2]{R}. Note that the framing
in loc. cit. is understood in the derived sense.
\end{remark}

{\begin{remark}
In the case when  $d_0\,=\,d_1\,=\,0$, $\theta\,=\,\OOO_D$, both $\JJJ^\theta$ and $\PPPb$ are group schemes,
obtained as commutative extensions of the abelian scheme $\JJJ_V$ by $(\CC^*)^{\delta-1}$. Such extensions
are classified by sections of the relative Ext-sheaf $\EXT^1_{V^{\mathrm{\acute et}}}(\JJJ_V,\,\GG_m^{\delta-1})$.
Denoting the respective sections of extension classes $\mathbf e_{\JJJ^\theta},\, \mathbf e_{\PPPb}$, one
can verify that $\mathbf e_{\PPPb}\,=\,\mathbf e_{\JJJ^\theta}^2$. It is highly likely that on the two
sides we have distinct nontrivial cohomology classes, so $\PPPb$ and $\JJJ^{\OOO_D}$ are supposedly non-isomorphic
symplectic extensions of $\JJJ_V$.
\end{remark}}

\section{Moduli space of framed Higgs bundles}

Let $X$ be an irreducible smooth projective curve, defined over $\mathbb C$,
of genus $g$, with $g\, \geq\, 2$. The canonical line bundle of $X$ will be denoted by
$K_X$. Fix a finite subset
$$
{\mathbb D} \,=\, \{x_1,\, \cdots,\, x_l \}\, \subset\, X.
$$
The divisor $x_1+\ldots + x_l $ on $X$ will also be denoted by $\mathbb D$. A Higgs field
on a holomorphic vector bundle $E\, \longrightarrow\, X$ is a section
$\varphi \,\in\, H^0(X,\, \text{End}(E)\otimes K_X
\otimes{\mathcal O}_X(\mathbb{D}))$. A Higgs bundle on $X$ is a
holomorphic vector bundle on $X$ equipped with a Higgs field.

Take a Higgs bundle $(E,\, \varphi)$ on $X$. The homomorphism $E\, \longrightarrow\,
E\otimes K_X \otimes{\mathcal O}_X(\mathbb{D})$ given by $\varphi$ will also be denoted by
$\varphi$. The Higgs bundle $(E,\, \varphi)$ is called \textit{stable} if for every subbundle
$0\, \not=\, F\, \subsetneq\, E$ with $\varphi(F)\, \subset\, F\otimes K_X \otimes
{\mathcal O}_X(\mathbb{D})$, the inequality
$$
\frac{\text{degree}(F)}{\text{rank}(F)}\ <\ \frac{\text{degree}(E)}{\text{rank}(E)}
$$
holds. For any fixed integer $r\, \geq\, 1$ and $d$, there is a moduli space ${\mathcal
M}_{\rm Higgs}(r)$ parametrizing the stable Higgs bundles of rank $r$ and degree $d$ on $X$
\cite{Hi}, \cite{Ni}. This moduli space ${\mathcal M}_{\rm Higgs}(r)$ is a smooth
irreducible quasiprojective variety over $\mathbb C$ of dimension $2(r^2(g-1)+1)$.

The above moduli space ${\mathcal M}_{\rm Higgs}(r)$ has a natural Poisson structure
\cite{Bo}, \cite{Ma}. This Poisson structure is constructed using a description of the
tangent bundle of ${\mathcal M}_{\rm Higgs}(r)$.

The total space of the line bundle $K_X \otimes{\mathcal O}_X(\mathbb{D})$ will be denoted by
$Y$. Consider the projectivization
\[
\widehat{p}\ :\ S \ :=\ \PP_X(\OOO_X\oplus K_X \otimes{\mathcal O}_X(\mathbb{D}))
\ \longrightarrow \ X.
\]
So we have $Y\ \subset\, S$ as a Zariski open subset; let
\[
D_\infty \ :=\ S\setminus Y \ \subset\ S
\]
be the complement, which is a divisor. The restriction of $\widehat{p}$ to $Y\ \subset\, S$ will
be denoted by $p$.

It can be shown that $S$ has a natural Poisson structure. Indeed, the relative anticanonical
bundle for the projection $\widehat{p}$ has the expression:
$$
K^{-1}_{\widehat{p}}\ =\ \widehat{p}^*(K_X \otimes{\mathcal O}_X(\mathbb{D}))\otimes {\mathcal O}_S(2
D_\infty).
$$
Hence we have $K^{-1}_S \,=\, K^{-1}_{\widehat{p}}\otimes \widehat{p}^* K^{-1}_X\,=\,
{\mathcal O}_S(2 D_\infty)\otimes \widehat{p}^*{\mathcal O}_X(\mathbb{D})$. So $S$ has a
natural Poisson structure whose singular locus is $D_\infty\cup \widehat{p}^{-1}(\mathbb{D})$.

Fix two integers $r$ and $d$, with $r\, \geq\, 1$. Let ${\mathcal M}^\circ_Y$ denote the moduli
space of stable sheaves ${\mathcal L}$ on $S$ satisfying the following conditions:
\begin{enumerate}
\item ${\mathcal L}$ is pure of dimension one,

\item the support of ${\mathcal L}$ is contained in $Y$,

\item the support of ${\mathcal L}$ is lies in the linear system $\widehat{p}^*
(K_X \otimes{\mathcal O}_X(\mathbb{D}))^{\otimes r}$, and

\item $\chi({\mathcal L}) \,=\, d+r(1-g)$, where $g\,=\, \text{genus}(X)$.
\end{enumerate}

The Poisson structure on $S$ produces a Poisson structure on ${\mathcal M}^\circ_Y$.

Note that the pullback $p^*(K_X \otimes{\mathcal O}_X(\mathbb{D}))$ has a tautological section
\begin{equation}\label{vp}
\varpi\, \in\, H^0(Y,\, p^*(K_X \otimes{\mathcal O}_X(\mathbb{D})))
\end{equation}
whose divisor on $S$ is $D_0- D_\infty$, where $D_0$ is the image of the zero section of
$K_X \otimes{\mathcal O}_X(\mathbb{D})$, and any ${\mathcal L}\, \in\, {\mathcal M}^\circ_Y$
has a morphism
\begin{equation}\label{vp2}
{\mathcal T}_{\mathcal L}\, :\, {\mathcal L}\, \longrightarrow\, {\mathcal L}\otimes
(p^*(K_X \otimes{\mathcal O}_X(\mathbb{D})))
\end{equation}
given by tensoring with the section $\varpi$ in \eqref{vp}.

Let ${{\mathcal M}}^\circ_{\rm Higgs}(r)$ denote the moduli space of all stable Higgs bundles $(E,\,
\varphi)$ on $X$ of rank $r$ and degree $d$, where $\varphi\, \in\, H^0(X,\, \text{End}(E)\otimes K_X
\otimes{\mathcal O}_X(\mathbb{D}))$.

There is a natural isomorphism
\begin{equation}\label{vp3}
\Phi\, :\, {\mathcal M}^\circ_Y\, \longrightarrow\, {{\mathcal M}}^\circ_{\rm Higgs}(r)
\end{equation}
that sends any ${\mathcal L}\, \in\, {\mathcal M}^\circ_Y$ to the Higgs bundle
\begin{equation}\label{vp4}
(p_*{\mathcal L},\, p_*{\mathcal T}_{\mathcal L}),
\end{equation}
where ${\mathcal T}_{\mathcal L}$ is the homomorphism in \eqref{vp2}; note
that
$$
p_*({\mathcal L}\otimes(p^*(K_X \otimes{\mathcal O}_X(\mathbb{D}))))\,=\,
p_*({\mathcal L})\otimes K_X \otimes{\mathcal O}_X(\mathbb{D})
$$
by the projection formula.

As mentioned above, both ${\mathcal M}^\circ_Y$ and ${{\mathcal M}}^\circ_{\rm Higgs}(r)$ are equipped
with a Poisson structure. The isomorphism $\Phi$ in \eqref{vp3} is actually Poisson structure
preserving \cite{BBG}.

Let $\widetilde{\mathcal M}_Y$ denote the moduli space of all pairs $({\mathcal L},\, f)$,
where ${\mathcal L}\, \in\, {\mathcal M}^\circ_Y$ such that the support $C$ of $\mathcal L$
intersects the divisor $p^{-1}({\mathbb D}) \, \subset\, Y$ transversally, and $f$ is
a trivialization of $\mathcal L$ over the divisor $C\cap p^{-1}(\mathbb{D})\, \subset\, C$. Let
$$
\beta\,:\, \widetilde{\mathcal M}_Y\, \longrightarrow\,{\mathcal M}^\circ_Y
$$
be the forgetful map that sends any $({\mathcal L},\, f)$ to $\mathcal L$ by simply forgetting $f$.

Let $\widetilde{{\mathcal M}}_{\rm Higgs}(r)$ be the moduli space
of all triples of the form $(E,\, \varphi,\, \psi)$,
where $(E,\, \varphi)\, \in\, {{\mathcal M}}^\circ_{\rm Higgs}(r)$, and $\psi$ lies in a certain class of
trivializations of $E\big\vert_{\mathbb D}$. The class of trivializations in question consists
of all trivializations of $E\big\vert_{\mathbb D}$ that are compatible with $f$ via the isomorphism
$E\,=\,{p_C}_*{\mathcal L}$ (see \eqref{vp4}).

We recall the main result of \cite{LL}.

\begin{theorem}[{\cite{LL}}]\label{tll}
The moduli space $\widetilde{{\mathcal M}}_{\rm Higgs}(r)$ has a natural symplectic structure.
Furthermore, the forgetful map
$$
\pi\, :\, \widetilde{{\mathcal M}}_{\rm Higgs}(r)\, \longrightarrow\, {{\mathcal M}}^\circ_{\rm Higgs}(r)
$$
that sends any $(E,\, \varphi,\, \psi)$ to $(E,\, \varphi)$ by simply forgetting $\psi$
is actually a Poisson morphism.
\end{theorem}

There is an isomorphism
$$
\widetilde{\Phi}\, :\, \widetilde{\mathcal M}_Y\, \longrightarrow\, \widetilde{{\mathcal M}}_{\rm Higgs}(r)
$$
which fits in the following commutative diagram of maps:
\begin{equation}\label{j0}
\begin{matrix}
\widetilde{\mathcal M}_Y  & \stackrel{\widetilde{\Phi}}{\longrightarrow} & \widetilde{{\mathcal M}}_{\rm Higgs}(r)
\\
\,\,\, \Big\downarrow \beta && \,\,\, \Big\downarrow \pi\\
{\mathcal M}^\circ_Y & \stackrel{\Phi}{\longrightarrow} &  {{\mathcal M}}^\circ_{\rm Higgs}(r).
\end{matrix}
\end{equation}

In view of \eqref{j0}, from Theorem \ref{tll} we have the following:

\begin{theorem}\label{tll2}
The moduli space $\widetilde{\mathcal M}_Y$ has a natural symplectic structure, and the morphism
$\beta$ is a Poisson map.
\end{theorem}

\section*{Acknowledgements}
 
This work was supported by the PRCI SMAGP (ANR-20-CE40-0026-01) and the Labex CEMPI (ANR-11-LABX-0007-01). 
The first-named author is partially supported by a J. C. Bose Fellowship (JBR/2023/000003). 

\section*{Mandatory declarations}

No data were generated or used.


\begin{thebibliography}{MMM}

\bibitem[Ar]{A}
Artin, M.: Lectures on deformations of singularities, Tata Institute of Fundamental Research, Bombay, 1976.

\bibitem[BM]{BM}
Bruzzo, U. and Markushevich, D.:
{\em Moduli of framed sheaves on projective surfaces},
Documenta Math. {\bf 16} (2011), 399--410.

\bibitem[BPS]{BPS}
Bănică, C.; Putinar, M. and Schumacher, G.: {\em Variation der globalen Ext in
Deformationen kompakter komplexer Räume}, Math. Ann. {\bf 250} (1980),
135--155.

\bibitem[BLP]{BLP} Biswas, I., Logares, M. and Pe\'on-Nieto, A.: {\em Symplectic geometry
of a moduli space of framed Higgs bundles}, Int. Math. Res. Not. (2021),
5623--5650.

\bibitem[BBG]{BBG} Biswas, I., Bottacin, F. and G\'omez, T. L.: {\em Comparison of Poisson
structures on moduli spaces}, Revis. Mat. Complutense {\bf 36} (2023), 57--72.

\bibitem[Bo1]{Bo} Bottacin, F.: {\em Symplectic geometry on moduli spaces of stable pairs},
Ann. Sci. \'Ecole Norm. Sup. {\bf 28} (1995), 391--433.

\bibitem[Bo2]{Bot-1} Bottacin, F.: {\em
Poisson structures on moduli spaces of sheaves over Poisson surfaces},
Invent. Math. {\bf 121} (1995), 421--436.

\bibitem[Hi]{Hi} Hitchin, N. J.: {\rm The self--duality equations on a Riemann surface},\,
Proc. London Math. Soc.\, \textbf{55} (1987), 59--126.

{
\bibitem[HL1]{HL-1}
Huybrechts, D. and Lehn, M.: {\em Stable pairs on curves and surfaces}, J. Alg. Geom. {\bf 4} (1995), 67--104.

\bibitem[HL2]{HL-2}
Huybrechts, D.; Lehn, M.: {\em Framed modules and their moduli}, Internat. J. Math. {\bf 6} (1995), 297--324.  
}

\bibitem[HL3]{Hu}
Huybrechts, D. and Lehn, M.:
The geometry of moduli spaces of sheaves, Cambridge University Press,
Cambridge, second edition, 2010.
 
\bibitem[HM]{HM}
Hurtubise, J. C. and Markman, E.: {\em Elliptic Sklyanin integrable systems for arbitrary reductive
groups}, Adv. Theor. Math. Phys., {\bf 6} (2003), 873--978.

\bibitem[LL]{LL} Lee, J. C. and Lee, S.: {\em Semi-polarized meromorphic Hitchin and Calabi-Yau
integrable systems}, Int. Math. Res. Not. (2023), 9511--9564.

{
\bibitem[Le]{L}
Lehn, M.: Modulräume gerahmter Vektorbündel, Ph.D. thesis, Bonn, 1992. 
}

\bibitem[Ma]{Ma} Markman, E.: {\em Spectral curves and integrable systems},
Compositio Math. {\bf 93} (1994), 255--290.

\bibitem[MT]{MT}
Markushevich, D. and Tikhomirov, A. S.:
{\em New symplectic  $V$-manifolds of dimension four via the relative compactified Prymian},
Internat. J. Math. {\bf 18} (2007), 1187--1224.

\bibitem[Mu1]{Muk-1}
Mukai, S.:
{\em Symplectic structure of the moduli space of sheaves on an abelian or
  {$K3$} surface},
 Invent. Math. {\bf 77} (1984), 101--116.
 
 \bibitem[Mu2]{Muk-2}
Mukai, S.:
{\em Moduli of vector bundles on {$K3$} surfaces and symplectic manifolds},
 S\=ugaku, {\bf 39} (1987), 216--235, translated in: Sugaku Expositions {{\bf{1}}} (1988), no. 2, 139--174.

\bibitem[Ni]{Ni} Nitsure, N.: {\em Moduli space of semistable pairs on a curve},
Proc. London Math. Soc. {\bf 62} (1991), 275--300.

\bibitem[Rai]{R}
Rains, E.: {\em Birational morphisms and Poisson moduli spaces}, arXiv:1307.4032.

\bibitem[Ran]{Ran}
Ran, Z.:
{\em On the local geometry of moduli spaces of locally free sheaves}, Lecture Notes in Pure and Appl. Math., 179
Marcel Dekker, Inc., New York, 1996, 213--219.

\end{thebibliography}
\end{document}